\pdfoutput=1
\RequirePackage{ifpdf}
\ifpdf % We are running pdfTeX in pdf mode
\documentclass[pdftex]{sigma}
\else
\documentclass{sigma}
\fi

\newcommand\C{\mathbb{C}}
\newcommand\R{\mathbb{R}}
\newcommand\g{\mathfrak{g}}
\newcommand\m{\mathfrak{t}}
\newcommand\mm{\mathfrak{m}}
\newcommand\h{\mathfrak{h}}
\newcommand\K{\mathbb{K}}
\newcommand\D{\delta_{\mu_0}}

\numberwithin{equation}{section}

\newtheorem{Theorem}{Theorem}[section]
\newtheorem{Corollary}[Theorem]{Corollary}
\newtheorem{Lemma}[Theorem]{Lemma}
\newtheorem{Proposition}[Theorem]{Proposition}
 { \theoremstyle{definition}
\newtheorem{Definition}[Theorem]{Definition}
\newtheorem{Example}[Theorem]{Example}
\newtheorem{Remark}[Theorem]{Remark} }

\begin{document}
\allowdisplaybreaks

\newcommand{\arXivNumber}{1806.07553}

\renewcommand{\PaperNumber}{002}

\FirstPageHeading

\ShortArticleName{Coadjoint Orbits of Lie Algebras and Cartan Class}

\ArticleName{Coadjoint Orbits of Lie Algebras and Cartan Class}

\Author{Michel GOZE~$^\dag$ and Elisabeth REMM~$^\ddag$}

\AuthorNameForHeading{M.~Goze and E.~Remm}

\Address{$^\dag$~Ramm Algebra Center, 4 rue de Cluny, F-68800 Rammersmatt, France}
\EmailD{\href{mailto:goze.rac@gmail.com}{goze.rac@gmail.com}}

\Address{$^\ddag$~Universit\'e de Haute-Alsace, IRIMAS EA 7499, D\'epartement de Math\'ematiques,\\
\hphantom{$^\ddag$}~F-68100 Mulhouse, France}
\EmailD{\href{mailto:elisabeth.remm@uha.fr}{elisabeth.remm@uha.fr}}

\ArticleDates{Received September 13, 2018, in final form December 31, 2018; Published online January 09, 2019}

\Abstract{We study the coadjoint orbits of a Lie algebra in terms of Cartan class. In fact, the tangent space to a coadjoint orbit $\mathcal{O}(\alpha)$ at the point $\alpha$ corresponds to the characteristic space associated to the left invariant form~$\alpha$ and its dimension is the even part of the Cartan class of $\alpha$. We apply this remark to determine Lie algebras such that all the nontrivial orbits (nonreduced to a point) have the same dimension, in particular when this dimension is~$2$ or~$4$. We determine also the Lie algebras of dimension $2n$ or $2n+1$ having an orbit of dimension~$2n$.}


\Classification{17B20; 17B30; 53D10; 53D05}

\section{Introduction}

Let $G$ be a connected Lie group, $\g$ its Lie algebra and~$\g^*$ the dual vector space of~$\g$. We identify~$\g$ with the Lie algebra of left invariant vector fields on $G$ and $\g^*$ with the vector space of left invariant Pfaffian forms on $G$. The Lie group $G$ acts on $\g^*$ by the coadjoint action. If~$\alpha$ belongs to~$\g^*$, its coadjoint orbit $\mathcal{O}(\alpha)$ associated with this action is reduced to a point if~$\alpha$ is closed for the adjoint cohomology of~$\g$. If not, the coadjoint orbit is an even-dimensional manifold provided with a symplectic structure. From the Kirillov theory, if $G$ is a connected and simply connected nilpotent Lie group, there exists a canonical bijection from the set of coadjoint orbits onto the set of equivalence classes of irreducible unitary representations of this Lie group. In this work, we establish a link between the dimension of the coadjoint orbit of the form $ \alpha$ and $\operatorname{cl}(\alpha)$ its class in Elie Cartan's sense. More precisely $\dim \mathcal{O}(\alpha)= 2 \big[\frac{\operatorname{cl}(\alpha)}{2}\big]$. Recall that the Cartan class of $\alpha$ corresponds to the number of independent Pfaffian forms needed to define $\alpha$ and its differential ${\rm d}\alpha$ and it is equal to the codimension of the characteristic space \cite{AW,godbillon,GR-DGA}. The dimension of $\mathcal{O}(\alpha)$ results in a natural relation between this characteristic space and the tangent space at the point $\alpha$ to the orbit $\mathcal{O}(\alpha)$.

As applications, we describe classes of Lie algebras with additional properties related to its coadjoint orbits. For example, we determine all Lie algebras whose nontrivial orbits are all of dimension~$2$ or~$4$ and also the Lie algebras of dimension $2p$ or $2p+1$ admitting a maximal orbit of dimension $2p$ that is admitting $\alpha \in \g^*$ such that $\operatorname{cl}(\alpha) \geq 2p$.
Notice that the nilpotent Lie algebras classified in Sections~\ref{section3} and~\ref{section4} play a role in connection with a so-called inverse problem in representation theory of nilpotent Lie groups (see \cite{Beltita2} and \cite[Section~5.1]{Be2}).

\section{Dimension of coadjoint orbits and Cartan class}
\subsection{Cartan class of a Pfaffian form}
Let $M$ be a $n$-dimensional differentiable manifold and $\alpha$ a Pfaffian form on $M$, that is a differential form of degree $1$. The
{\it characteristic space of $\alpha$ at a point $x \in M$} is the linear subspace~$\mathcal{C}_x(\alpha)$ of the tangent space $T_x(M)$ of $M$ at the point $x$ defined by
\begin{gather*} \mathcal{C}_x(\alpha)=A(\alpha(x)) \cap A({\rm d}\alpha(x)),\end{gather*}
where \begin{gather*} A(\alpha(x))=\{X_x \in T_x(M), \, \alpha(x)(X_x)=0\}\end{gather*} is the {\it associated subspace of $\alpha (x)$},
\begin{gather*} A({\rm d}\alpha(x))= \{X_x \in T_x(M), \, {\rm i} (X_x){\rm d} \alpha(x) =0 \}\end{gather*} is the {\it associated subspace of ${\rm d}\alpha (x)$} and ${\rm i}(X_x){\rm d}\alpha(x)(Y_x)={\rm d}\alpha(x)(X_x,Y_x)$.

\begin{Definition}Let $\alpha$ be a Pfaffian form on the differential manifold $M$. The {\it Cartan class of~$\alpha$ at the point $x \in M$}~\cite{Cartan} is the codimension of the characteristic space~$\mathcal{C}_x(\alpha)$ in the tangent space~$T_x(M)$ to~$M$ at the point~$x$. We denote it by $\operatorname{cl}(\alpha)(x)$.
\end{Definition}

 The function $x \rightarrow \operatorname{cl}(\alpha)(x)$ is with positive integer values and is lower semi-continuous, that is, for every $x \in M$ there exists a suitable neighborhood $V$ such that for every $x_1 \in V$ one has $\operatorname{cl}(\alpha)(x_1) \geq \operatorname{cl}(\alpha)(x)$.

The {\it characteristic system of $\alpha$ at the point $x$ of $ M$} is the subspace $ \mathcal{C}_x^*(\alpha)$ of the dual $T_x^*(M)$ of $T_x(M)$ orthogonal to $\mathcal{C}_x(\alpha)$:
\begin{gather*} \mathcal{C}_x^*(\alpha)=\{\omega(x) \in T_x^*(M), \, \omega(x)(X_x)=0, \, \forall\, X_x \in \mathcal{C}_x(\alpha)\}.\end{gather*}
Then
\begin{gather*}\operatorname{cl}(\alpha)(x)= \dim \mathcal{C}_x^*(\alpha).\end{gather*}

\begin{Proposition}
If $\alpha$ is a Pfaffian form on $M$, then
\begin{itemize}\itemsep=0pt
\item $\operatorname{cl}(\alpha)(x)= 2p+1$ if $(\alpha \wedge ({\rm d}\alpha)^p)(x) \neq 0$ and $ ({\rm d}\alpha)^{p+1}(x)=0$,
\item $\operatorname{cl}(\alpha)(x)=2p$ if $({\rm d}\alpha )^p(x) \neq 0 $ and $ (\alpha \wedge ({\rm d}\alpha )^{p})(x)=0$.
\end{itemize}
\end{Proposition}
In the first case, there exists a basis $\{\omega_1(x)=\alpha(x),\omega_2(x),\dots,\omega_n(x)\}$ of $T_x^*(M)$ such that
\begin{gather*} {\rm d}\alpha(x)=\omega_2(x)\wedge \omega_3(x)+ \dots +\omega_{2p}(x)\wedge \omega_{2p+1}(x)\end{gather*}
and
\begin{gather*} \mathcal{C}_x^*(\alpha)=\R\{\alpha(x)\}+A^*({\rm d}\alpha (x)).\end{gather*}
In the second case, there exists a basis $\{\omega_1(x)=\alpha(x),\omega_2(x),\dots,\omega_n(x)\}$ of $T_x^*(M)$ such that
\begin{gather*} {\rm d}\alpha(x)=\alpha(x)\wedge \omega_2(x)+ \dots +\omega_{2p-1}(x)\wedge \omega_{2p}(x)\end{gather*}
and
\begin{gather*}\mathcal{C}_x^*(\alpha)=A^*({\rm d}\alpha (x)).\end{gather*}

If the function $\operatorname{cl}(\alpha)(x)$ is constant, that is, $\operatorname{cl}(\alpha)(x)= \operatorname{cl}(\alpha)(y)$ for any $x, y \in M$, we say that the Pfaffian form $\alpha$ is of constant class and we denote by $\operatorname{cl}(\alpha)$ this constant. The distribution
\begin{gather*} x \rightarrow \mathcal{C}_x(\alpha)\end{gather*}
is then regular and it is an integrable distribution of dimension $n-\operatorname{cl}(\alpha)$, called the characteristic distribution of~$\alpha$. It is equivalent to say that the Pfaffian system
\begin{gather*} x \rightarrow \mathcal{C}^*_x(\alpha)\end{gather*}
is integrable and of dimension $\operatorname{cl}(\alpha)$.

If $M=G$ is a connected Lie group, we identify its Lie algebra $\g$ with the space of left invariant vector fields and its dual $\g^*$ with the space of left invariant Pfaffian forms. Then if $\alpha \in \g^*$, the differential ${\rm d}\alpha$ is the $2$-differential left invariant form belonging to $\Lambda^2(\g^*)$ and defined by
\begin{gather*}{\rm d}\alpha(X,Y)=-\alpha[X,Y]\end{gather*}
for any $X,Y \in \g$. It is obvious that any left invariant form $\alpha \in \g^*$ is of constant class and we will speak on the {\it Cartan class $\operatorname{cl}(\alpha)$ of a linear form $\alpha \in \g^*$}. We have
\begin{itemize}\itemsep=0pt
\item $\operatorname{cl}(\alpha)=2p+1$ if and only if $\alpha \wedge ({\rm d}\alpha)^p\neq 0$ and $ ({\rm d}\alpha)^{p+1}=0$,
\item $\operatorname{cl}(\alpha)=2p$ if and only if $({\rm d}\alpha )^p \neq 0 $ and $ \alpha \wedge ({\rm d}\alpha )^{p}=0$.
\end{itemize}

\begin{Definition}Let $\g$ be an $n$-dimensional Lie algebra.
\begin{itemize}\itemsep=0pt
\item It is called {\it contact Lie algebra} if $n=2p+1$ and if there exists a contact linear form, that is, a linear form of Cartan class equal to $2p+1$.
\item It is called {\it Frobenius Lie algebra} if $n=2p$ and if there exists a Frobenius linear form, that is, a linear form of Cartan class equal to $2p$.
\end{itemize}
\end{Definition}

If $\alpha \in \g^*$ is neither a contact nor a Frobenius form, the characteristic space $\mathcal{C}(\alpha)=\mathcal{C}_e(\alpha)$ at the unit~$e$ of~$G$ is not trivial and the characteristic distribution on $G$ given by $\mathcal{C}_x(\alpha)$ with $x \in G$ has a constant non-zero dimension. As it is integrable, the subspace $\mathcal{C}(\alpha)$ is a Lie subalgebra of~$\g$.

\begin{Proposition}\label{vergne}Let $\alpha \in \g^*$ be a linear form of maximal class, that is
\begin{gather*} \forall\, \beta \in \g^*, \qquad \operatorname{cl}(\alpha) \geq \operatorname{cl}(\beta).\end{gather*}
Then $\mathcal{C}(\alpha)$ is an abelian subalgebra of~$\g$.
\end{Proposition}

\begin{proof} Assume that $\alpha$ is a linear form of maximal class. If $\operatorname{cl}(\alpha)=2p+1$, there exists a basis $\{\omega_1=\alpha,\omega_2,\dots,\omega_n\}$ of $\g^*$ such that ${\rm d}\alpha=\omega_2 \wedge \omega_3 + \dots + \omega_{2p} \wedge \omega_{2p+1}$ and $\mathcal{C}^*(\alpha)$ is generated by $\{\omega_{1},\dots,\omega_{2p+1}\}$. If the subalgebra $\mathcal{C}(\alpha)$ is not abelian, there exists $j$, $2p+2 \leq j \leq n$ such that
\begin{gather*} {\rm d}\omega_j \wedge \omega_{1} \wedge \dots \wedge \omega_{2p+1} \neq 0.\end{gather*}
 Then
\begin{gather*}\operatorname{cl}(\alpha+t\omega_j) > \operatorname{cl}(\alpha)\end{gather*}
for some $t \in \R$. But $\alpha $ is of maximal class. Then $\mathcal{C}(\alpha)$ is abelian. The proof when $\operatorname{cl}(\alpha)=2p$ is similar.
\end{proof}

Note that this result can alternatively derived from the result of M.~Duflo and M.~Vergne~\cite{Vergne1}.

Recall some properties of the class of a linear form on a Lie algebra. The proofs of these statements are given in \cite{GozeEllipse, GozeCras1, GR-DGA}.
\begin{itemize}\itemsep=0pt
 \item If $\g$ is a finite-dimensional nilpotent Lie algebra, then the class of any non-zero $\alpha \in \g^*$ is always odd.
 \item A real or complex finite-dimensional nilpotent Lie algebra is never a Frobenius Lie algebra. More generally, an unimodular Lie algebra is non-Frobenius~\cite{Diatta1, Ooms}.
 \item Let $\g$ be a real compact Lie algebra. Any nontrivial $\alpha \in \g^*$ has an odd Cartan class.
 \item Let $\g$ be a complex semisimple Lie algebra of rank $r$. Then any $\alpha \in \g^*$ satisfies $\ \operatorname{cl}(\alpha) \leq n-r+1$~\cite{GozeCras1}. In particular, a semisimple Lie algebra is never a Frobenius algebra. A~semisimple Lie algebra is a contact Lie algebra if its rank is~$1$, that is, $\g$ is isomorphic to~$\mathfrak{sl}(2,\C)$. In particular, if $\g$ is a contact real simple Lie algebra, then~$\g$ is isomorphic to~$\mathfrak{sl}(2,\R)$ or $\mathfrak{so}(3)$.
 \item The Cartan class of any linear nontrivial form on a simple nonexceptional complex Lie algebra of rank $r$ satisfies $\operatorname{cl}(\alpha) \geq 2r$~\cite{GozeCras1}. Moreover, if $\g$ is isomorphic of type~$A_r$, there exists a linear form of class~$2r$ which reaches the lower bound.
\item Any $(2p+1)$-dimensional contact real Lie algebra such that any nontrivial linear form is a~contact form is isomorphic to $\mathfrak{so}(3)$.
\end{itemize}

\subsection{Cartan class and the index of a Lie algebra} For any $\alpha \in \g^*$, we consider the stabilizer $\g_{\alpha }=\{X\in \g,\, \alpha\circ \operatorname{ad} X=0\}$ and $d$ the minimal dimension of $\g_{\alpha }$ when $\alpha$ lies in $\g^*$. It is an invariant of $\g$ which is called the index of $\g$. If $\alpha$ satisfies $\dim \g_{\alpha }=d$ then, from Proposition~\ref{vergne}, $\g_{\alpha }$ is an abelian subalgebra of $\g$. Considering the Cartan class of $\alpha$, $\g_{\alpha }$ is the associated subspace of ${\rm d}\alpha$:
\begin{gather*}\g_{\alpha }=A({\rm d}\alpha),\end{gather*}
 so the minimality is realized by a form of maximal class and we have $d=n-\operatorname{cl}(\alpha)+1$ if the Cartan class $\operatorname{cl}(\alpha)$ is odd or $d=n-\operatorname{cl}(\alpha)$ if $\operatorname{cl}(\alpha)$ is even. In particular,
 \begin{enumerate}\itemsep=0pt
 \item[1)] if $\g$ is a $2p$-dimensional Frobenius Lie algebra, then the maximal class is $2p$ and $d=0$,
 \item[2)] if $\g$ is $(2p+1)$-dimensional contact Lie algebra, then $d=n-n+1=1$.
\end{enumerate}
This relation between index and Cartan class is useful to compute sometimes this index. For example we have
\begin{Proposition}
Let $L_n$ or $Q_{2p}$ be the naturally graded filiform Lie algebras. Their index satisfy
\begin{enumerate}\itemsep=0pt
 \item[$1)$] $d(L_n)=n-2$,
 \item[$2)$] $d(Q_{2p})=2$.
\end{enumerate}
\end{Proposition}
In fact $L_n$ is defined in a basis $\{e_0,\dots,e_{n-1}\}$ by $[e_0,e_i]=e_{i+1}$ for $i=1,\dots,n-2$ and we have $\operatorname{cl}(\alpha) \in \{1,3\}$ and $d=n-2$ for any $\alpha \in L_n^*$. The second algebra $Q_{2p}$ is defined in the basis $\{e_0,\dots , e_{2p-1}\}$ by $[e_0,e_i]=e_{i+1}$ for $i=1,\dots,2p-3$, $[e_{i},e_{2p-1-i}]=(-1)^{i-1}e_{2p-1}$ for $i=1,\dots,p-1$. In this case we have $\operatorname{cl}(\alpha) \in \{1,3,2p-1\}$ for any $\alpha \in Q_{2p}^*$ and $d=2$. Let us note that a direct computation of these indexes are given in~\cite{Adimi}.

\subsection{The coadjoint representation}

Let $G$ be a connected Lie group and $\g$ its (real) Lie algebra. The adjoint representation of $G$ on~$\g$ is the homomorphism of groups:
\begin{gather*}\operatorname{Ad}\colon \ G \rightarrow \operatorname{Aut}(\g)\end{gather*}
defined as follows. For every $x \in G$, let $A(x)$ be the automorphism of $G$ given by $A(x)(y)=xyx^{-1}$. This map is differentiable and the tangent map to the identity~$e$ of $G$ is an automorphism of~$\g$. We denote it by $\operatorname{Ad}(x)$.

\begin{Definition} The coadjoint representation of $G$ on the dual $\g^*$ of $\g$ is the homomorphism of groups:
\begin{gather*}\operatorname{Ad}^*\colon \ G \rightarrow \operatorname{Aut}( \g^*)\end{gather*}
defined by
\begin{gather*}\big\langle \operatorname{Ad}^*(x)\alpha,X\big\rangle=\big\langle \alpha,\operatorname{Ad}\big(x^{-1}\big)X\big\rangle\end{gather*}
for any $\alpha \in \g^*$ and $X \in \g$.
\end{Definition}
The coadjoint representation is sometimes called the $K$-representation. For $\alpha \in \g^*$, we denote by $\mathcal{O}(\alpha)$ its orbit, called the coadjoint orbit, for the coadjoint representation. The following result is classical~\cite{Ki}: any coadjoint orbit is an even-dimensional differentiable manifold endowed with a symplectic form.

Let us compute the tangent space to this manifold $\mathcal{O}(\alpha)$ at the point~$\alpha$. Any $\beta \in \mathcal{O}(\alpha)$ is written $\beta = \alpha \circ \operatorname{Ad}\big(x^{-1}\big)$ for some $x \in G$. The map $\rho\colon G \rightarrow \g^*$ defined by $\rho(x)=\alpha\circ \operatorname{Ad}(x)$ is differentiable and its tangent map at the identity of~$G$ is given by $\rho^T_e(X)={\rm i}(X){\rm d}\alpha$ for any $X \in \g$ with ${\rm i}(X){\rm d}\alpha(Y)=-\alpha [X,Y]$. Then the tangent space to $\mathcal{O}(\alpha)$ at the point $\alpha$ corresponds to $A^*({\rm d}\alpha)=\{\omega \in \g^*, \, \omega(X)=0 ,\, \forall\, X \in A({\rm d}\alpha)\}$.

\begin{Proposition}Consider a non-zero $\alpha \in \g^*$. The tangent space to $\mathcal{O}(\alpha)$ at the point $\alpha$ is isomorphic to the dual space $A^*({\rm d}\alpha)$ of the associated space $A({\rm d}\alpha)$ of ${\rm d}\alpha$.
\end{Proposition}
\begin{Corollary}
Consider a non-zero $\alpha$ in $\g^*$. Then $\dim \mathcal{O}(\alpha)=2p$ if and only if $\operatorname{cl}(\alpha)=2p$ or $\operatorname{cl}(\alpha)=2p+1$.
\end{Corollary}
An immediate application is
\begin{Proposition}\label{simple}
Any $(2p+1)$-dimensional Lie algebra with $\dim \mathcal{O}(\alpha)=2p$ for all $\alpha \in \g ^* \setminus \{ 0\}$, is isomorphic to $\mathfrak{so}(3)$ or $\mathfrak{sl}(2,\R)$ and then of dimension $3$.
\end{Proposition}
\begin{proof} From \cite{GozeCras1, GR-DGA} any contact Lie algebra such that any nontrivial linear form is a contact form is isomorphic to $\mathfrak{so}(3)$. Assume now that any nontrivial form on $\g$ is of Cartan class equal to $2p$ or $2p+1$. With similar arguments developed in \cite{GozeCras1, GR-DGA}, we prove that such a Lie algebra is semisimple. But in this case, we have seen that $\g$ is isomorphic to $\mathfrak{sl}(2,\R)$ or $\mathfrak{so}(3)$.
\end{proof}

\begin{Remark} Assume that $\dim \g=2p$ and all the nontrivial coadjoint orbits are also of dimension~$2p$. Then for any nontrivial $\alpha \in \g^*$, $\operatorname{cl}(\alpha)=2p$. If $I$ is a nontrivial abelian ideal of~$\g$, there exists $\omega \in \g^*$, $\omega \neq 0$ such that $\omega(X)=0$ for any $X \in I$. The Cartan class of this form~$\omega$ is smaller than~$2p$. Then $\g$ is semisimple. But the behavior of the Cartan class on simple Lie algebra leads to a contradiction.
\end{Remark}

We deduce also from the previous corollary:
\begin{Proposition}\quad
\begin{enumerate}\itemsep=0pt
\item[$1.$] If $\g$ is isomorphic to the $(2p+1)$-dimensional Heisenberg algebra, then any nontrivial coadjoint orbit is of dimension $2p$.
\item[$2.$] If $\g$ is isomorphic to the graded filiform algebra $L_n$, then any nontrivial coadjoint orbit is of dimension~$2$.
\item[$3.$] If $\g$ is isomorphic to the graded filiform algebra $Q_n$, then any nontrivial coadjoint orbit is of dimension~$2$ or~$n-2$.
\item[$4.$] If $\g$ is a $(2p+1)$-dimensional $2$-step nilpotent Lie algebra with a coadjoint orbit of dimen\-sion~$2p$, then $\g$ is isomorphic to the $(2p+1)$-dimensional Heisenberg Lie algebra $($see~{\rm \cite{GR-DGA})}.
\item[$5.$] If $\g$ is a $n$-dimensional complex classical simple Lie algebra of rank $r$, then the maximal dimension of the coadjoint orbits is equal to $n-r$ if this number is even, if not to $n-r-1$ $($see~{\rm \cite{GozeCras1})}.
 \end{enumerate}
 \end{Proposition}

\section{Lie algebras whose coadjoint orbits are of dimension 2 or 0}\label{section3}

In this section, we determine all Lie algebras whose coadjoint orbits are of dimension $2$ or $0$. This problem was initiated in \cite{Arnal,Beltita}. This is equivalent to say that the Cartan class of any linear form is smaller or equal to $3$. If $\g$ is a Lie algebra having this property, any direct product $\g_1=\g \bigoplus J$ of $\g$ by an abelian ideal $J$ satisfies also this property. We shall describe these Lie algebras up to an abelian direct factor, that is indecomposable Lie algebras. It is obvious that for any Lie algebra of dimension $2$ or $3$, the dimensions of the coadjoint orbits are equal to~$2$ or~$0$. We have also seen:
\begin{Proposition}Let $\g$ be a simple Lie algebra of rank $1$. Then for any $\alpha \neq 0 \in \g^*$, \mbox{$\dim \mathcal{O}(\alpha)=2$}. Conversely, if~$\g$ is a Lie algebra such that $\dim \mathcal{O}(\alpha)=2$ for any $\alpha \neq 0 \in \g^*$ then~$\g$ is simple of rank $1$.
\end{Proposition}
Recall that the rank of a real semisimple Lie algebra $\g$ is the dimension of any Cartan subalgebra $\h$ of $\g$. This is well defined since $\h$ is a Cartan subalgebra if and only if $\h_{\C}$ is Cartan in the complexified simple Lie algebra $\g_{\C}$. Now we examine the general case. Assume that~$\g$ is a Lie algebra of dimension greater or equal to~$4$ such that for any nonzero $\alpha \in \g^*$ we have $\operatorname{cl}(\alpha)=3,2$ or~$1$.

Assume in a first step that $\operatorname{cl}(\alpha)=2$ for any nonclosed $\alpha \in \g^*$. Let $\alpha$ be a non-zero $1$-form and $\operatorname{cl}(\alpha)=2$. Then there exists a basis $\{\alpha= \omega_1,\dots,\omega_n\}$ of $\g^*$ such that ${\rm d}\alpha={\rm d}\omega_1=\omega_1\wedge \omega_2$. This implies ${\rm d}({\rm d}\omega_1)=0= - \omega_1\wedge {\rm d}\omega_2$. Therefore ${\rm d}\omega_2=\omega_1 \wedge \omega$ with $\omega \in \g ^*$. As $\operatorname{cl}(\omega_2)\leq 2$, $\omega_2 \wedge \omega=0$. If $\{ X_1, X_2, \dots , X_n\}$ is the dual basis of $\{\alpha= \omega_1,\dots,\omega_n\}$, then $A({\rm d}\omega_1)=\K\{ X_3, \dots, X_n \} $ is an abelian subalgebra of $\g$. Suppose $ [X_1,X_2]=X_1+aX_2+ U$ where $U,X_1,X_2$ are linearly independent. The dual form of~$U$ would be of class 3 so $U=0$. Under the assumption of change of basis if $a\neq 0$ we can assume that $[X_1,X_2]=X_1$. So ${\rm d}\omega_2=\omega_1\wedge \omega$ and $\omega \wedge \omega_2=0$ imply that ${\rm d}\omega_2=0$. This implies that $A({\rm d}\omega_1)$ is an abelian ideal of codimension $2$. Let $\beta$ be in $A({\rm d}\omega_1)^*$, with ${\rm d}\beta\neq 0$. If such a form doesn't exist then $\K X_1 \oplus A({\rm d}\omega_1)$ is an abelian ideal of codimension~$1$. Otherwise ${\rm d}\beta=\omega_1\wedge \beta_1+\omega_2 \wedge \beta_2$ with $\beta_1,\beta_2 \in A({\rm d}\omega_1)$. As $\beta \wedge {\rm d}\beta={\rm d}\beta^2=0$, $\beta_1 \wedge \beta_2=0$ which implies ${\rm d}\beta=(a\omega_1+b\omega_2)\wedge \beta$. But ${\rm d}({\rm d}\beta)=0$ implies $a {\rm d}\omega_1\wedge \beta=0=a\omega_1\wedge \omega_2\wedge \beta$ thus $a=0$ and ${\rm d}\beta=b\omega_2 \wedge \beta$. We conclude that $[X_1, A({\rm d}\omega_1)] =0$ and $\K\{X_1\} \oplus A({\rm d}\omega_1)$ is an abelian ideal of codimension~1.

Assume now that there exists $\omega$ of class $3$. There exists a basis $\mathcal{B}=\{\omega_1,\omega_2, \omega_3=\omega , \dots , \omega_n\}$ of $\g^*$ such as ${\rm d}\omega ={\rm d}\omega_3=\omega_1\wedge \omega_2$ and the subalgebra $A({\rm d}\omega_3)=\K\{X_3,\dots,X_n\}$ is abelian. If $\{ X_1, \dots, X_n \}$ is the dual basis of $\mathcal{B}$, we can assume that $[X_1,X_2]=X_3. $ As $A({\rm d}\omega_3)$ is an abelian subalgebra of $\g$, for any $\alpha \in \g^*$ we have ${\rm d}\alpha= \omega_1 \wedge \alpha_1 + \omega_2 \wedge \alpha_2$ with $\alpha_1, \alpha_2 \in A({\rm d}\omega_3)^*$. But $\operatorname{cl}(\alpha) \leq 3$. Therefore $\omega_1 \wedge \alpha_1 \wedge \omega_2 \wedge \alpha_2=0$ which implies that $\alpha_1 \wedge \alpha_2=0$. So for any $\alpha \in A({\rm d}\omega_3)^*$ there exist $\alpha_1 \in A({\rm d}\omega_3)^*$ such that ${\rm d}\alpha=(a\omega_1+b\omega_2)\wedge \alpha_1$. Since $\g$ is indecomposable, for any $X \in A({\rm d}\omega_3)$ and $X \notin \mathcal{D}(\g)$, there exists $X_{12} \in \R\{ X_1, X_2 \} $ such that $[X_{12}, X]\neq 0$. We deduce

\begin{Proposition}\label{p8}
Let $\g$ an indecomposable Lie algebra such that the dimension of the nontrivial coadjoint orbits is $2$. We suppose that there exists $\omega \in \g^ *$ such that $\operatorname{cl}(\omega)=3$. If $n \ge 7$ then $\g=\m \oplus I_{n-1}$ where $I_{n-1}$ is an abelian ideal of codimension 1 and $\m$ a $1$-dimensional Lie subalgebra of $\g$.
\end{Proposition}

It remain to study the particular cases of dimension $4$, $5$ and $6$. The previous remarks show that:
\begin{itemize}\itemsep=0pt
\item If $\dim \g=4$ then $\g $ is isomorphic to one of the following Lie algebra given by its Maurer--Cartan equations
\begin{gather*}
\begin{cases}
{\rm d}\omega_3=\omega_1 \wedge \omega_2 , \\
{\rm d}\omega_1=\omega_1 \wedge \omega_4, \\
{\rm d}\omega_2=- \omega_2 \wedge \omega_4 , \\
{\rm d}\omega_4=0,
\end{cases}
\qquad
\begin{cases}
{\rm d}\omega_3=\omega_1 \wedge \omega_2, \\
{\rm d}\omega_1=\omega_2 \wedge \omega_4 , \\
{\rm d}\omega_2=- \omega_1 \wedge \omega_4, \\
{\rm d}\omega_4=0,
\end{cases}
\qquad \m\oplus I_3,
\end{gather*}
where $I_3$ is an abelian ideal of dimension $3$.

\item If $\dim \g=5$ then $\g $ is isomorphic to one of the following Lie algebra
\begin{gather*}
\begin{cases}
{\rm d}\omega_3=\omega_1 \wedge \omega_2 ,\\
{\rm d}\omega_1={\rm d}\omega_2= 0 ,\\
{\rm d}\omega_4=\omega_1 \wedge \omega_3 ,\\
{\rm d}\omega_5= \omega_2 \wedge \omega_3,
\end{cases}
\qquad \m\oplus I_4,
\end{gather*}
where $I_4$ is an abelian ideal of dimension $4$.

\item If $\dim \g=6$ then $\g $ is isomorphic to one of the following Lie algebra
\begin{gather*}
\begin{cases}
{\rm d}\omega_3=\omega_1 \wedge \omega_2 ,\\
{\rm d}\omega_1={\rm d}\omega_2={\rm d}\omega_4= 0 ,\\
{\rm d}\omega_5=\omega_2 \wedge \omega_4 ,\\
{\rm d}\omega_6= \omega_1 \wedge \omega_4,
\end{cases}
\qquad \m\oplus I_5,
\end{gather*}
where $I_5$ is an abelian ideal of dimension $5$.
\end{itemize}

\begin{Remark}\quad\begin{enumerate}\itemsep=0pt
\item Among the Lie algebras $\g=\m \oplus I_{n-1}$ where $I_{n-1}$ is an abelian ideal of dimension $n-1$ and~$\m$ a~$1$-dimensional Lie subalgebra of~$\g$, there exist a family of nilpotent Lie algebras which are the ``model'' for a given characteristic sequence (see~\cite{GR-Kegel, RBreadth}). They are the nilpotent Lie algebras~$L_{n,c}$, $c\in \{(n-1,1),(n-3,2,1),\dots,(2,1,\dots,1)\}$ defined by
\begin{gather}
[U,X_1]=X_2, \ [U,X_2]=X_3, \ \dots, \ [U,X_{n_1-1}]=X_{n_1}, \ [U,X_{n_1}]=0,\nonumber\\
[U,X_{n_1+1}]=X_{n_1+2}, \ [ U,X_{n_1+2}]=X_{n_1+3}, \ \dots, \ [U,X_{n_2-1}]=X_{n_2}, \ [U,X_{n_2}]=0,\nonumber\\
\cdots\cdots\cdots\cdots\cdots\cdots\cdots\cdots\cdots\cdots\cdots\cdots\cdots\cdots\cdots\cdots\cdots\cdots\cdots\cdots\cdots\cdots\nonumber \\
[ U,X_{n_{k-2}+1}]=X_{n_{k-2}+2}, \ \dots, \ [ U,X_{n_{k-1}-1}]=X_{n_{k-1}}, \ [U,X_{n_{k-1}}]=0.\label{Lc}
\end{gather}
The characteristic sequence $c$ corresponds to $c(U)$ and $\{X_1,\dots,X_{n_{k-1}}\}$ is a Jordan basis of $\operatorname{ad} U$. We shall return to this notion in the next section.

\item Let $U(\g)$ be the universal enveloping algebra of $\g$ and consider the category $U(\g)-\mathcal{M}{\rm od}$ of right $U(\g)$-module. Then if~$\g$ is a Lie algebra described in this section (that is with coadjoint orbits of dimension 0 or 2) thus any $U(\g)$-mod satisfy the property that ``any injective hulls of simple right $U(\g)$-module are locally Artinian'' (see~\cite{Ha}).

\item The notion of elementary quadratic Lie algebra was introduced by G.~Pinczon and R.~Ushirobira~\cite{Pin}. They also prove that is $\g$ is an elementary quadratic Lie algebra, then all coadjoint orbits have dimension at most~$2$.
\end{enumerate}
\end{Remark}

\section[Lie algebras whose nontrivial coadjoint orbits are of dimension 4]{Lie algebras whose nontrivial coadjoint orbits\\ are of dimension 4}\label{section4}

We generalize some results of the previous section, considering here real Lie algebras such that for a fixed $p \in \mathbb{N}$, $\dim\mathcal{O}(\omega)=2p$ or $0$ for any $\omega \in \g^*$. We are interested, in this section, in the case $p=2$. The Cartan class of any nonclosed linear form is equal to~$5$ or~$4$.

\begin{Lemma}\label{o4} Let $\g$ be a Lie algebra whose Cartan class of any nontrivial and nonclosed linear form is $4$ or $5$. Then $\g$ is solvable.
\end{Lemma}
\begin{proof} If $\g$ is a simple Lie algebra of rank $r$ and dimension $n$, then the Cartan class of any linear form $\omega \in \g^*$ satisfies $c \leq n-r+1$ and this upper bound is reached. Then $n-r+1=4$ or $5$ and the only possible case is for $r=2$ and $\g=\mathfrak{so}(4)$. Since this algebra is compact, the Cartan class is odd. We can find a basis of $\mathfrak{so}(4)$ whose corresponding Maurer--Cartan equations are
\begin{gather*}
 \begin{cases}
{\rm d}\omega_1=-\omega_2 \wedge\omega_4-\omega_3\wedge\omega_5, \\
{\rm d}\omega_2=\omega_1 \wedge\omega_4-\omega_3\wedge\omega_6, \\
{\rm d}\omega_3=\omega_1 \wedge\omega_5+\omega_2\wedge\omega_6, \\
{\rm d}\omega_4=-\omega_1 \wedge\omega_2-\omega_5\wedge\omega_6, \\
{\rm d}\omega_5=-\omega_1 \wedge\omega_3+\omega_4\wedge\omega_6, \\
{\rm d}\omega_6=-\omega_2 \wedge\omega_3-\omega_4\wedge\omega_5.
\end{cases}
\end{gather*}
If each of the linear forms of this basis has a Cartan class equal to~$5$, it is easy to find a linear form, for example $\omega_1+ \omega_6$, of Cartan class equal to $3$. Then $\g$ is neither simple nor semisimple. This implies also that the Levi part of a nonsolvable Lie algebra is also trivial, then~$\g$ is solvable.
\end{proof}

\subsection{Description of these Lie algebras}

A consequence of Lemma~\ref{o4} is that $\g$ contains a nontrivial abelian ideal. From the result of the previous section, the codimension of this ideal~$I$ is greater or equal to~$2$ and $\g=\mm \oplus I$ where~$\mm$ is a vector subspace of $\g$ of dimension greater or equal to~$2$.

Assume in a first time that $\dim \mm=2$. Since $I$ is a maximal abelian ideal, the Cartan class of any nontrivial linear form is $4$ or $5$ and the coadjoint nontrivial orbits are of dimension~$4$. To precise this case it remains to describe the action of $\mm$ on $I$ when we consider the second case. Assume that $\g=\mm \oplus I$ and $\dim \mm =2$. Let $\{T_1,T_2\}$ be a basis of $\mm$. Then $\widetilde{\g}=\g/\K\{T_2\}$ is a~Lie algebra whose any nonclosed linear form is of class $2$ or $3$. Such Lie algebra is described in Proposition~\ref{p8}.

\begin{Proposition} \label{c4} Let $\g$ be a Lie algebra, $\g=\mm \oplus I$ where $I$ is an abelian ideal of codimension~$2$ and whose coadjoint orbit of any nonclosed linear form is of dimension $4$. Let $\{T_1,T_2\}$ be a~basis of $\mm$ and $\widetilde{\g}=\g/\K\{T_2\}$. Then $\g$ is a~one-dimensional extension by a derivation~$f$ of~$\widetilde{\g}$ such that $f(T_1)=0$, $\operatorname{Im}(f) =\operatorname{Im} (\operatorname{ad} T_1)$ and for any $Y \in \operatorname{Im}(\operatorname{ad} T_1)$, there exist $X_1,X_2 \in I$ linearly independent such that
\begin{gather*}f(X_2)=[T_1,X_1]=Y.\end{gather*}
\end{Proposition}

Assume now that $\dim \mm =3$. In this case $\g/I$ is abelian, that is $[\mm,\mm] \subset I$. Otherwise, there would exist a linear form on $\mm^*$ of Cartan class equal to~$2$ or~$3$. This implies that for any $\omega \in \mm^*$ we have ${\rm d}\omega=0$ and we obtain, considering the dimension of $[\mm,\mm]$, the following Lie algebras which are nilpotent because the Cartan class is always odd:
 \begin{enumerate}\itemsep=0pt
 \item $\dim [\mm,\mm]=3$:
 \begin{gather}\label{n72}
 \begin{cases}
 {\rm d}\omega_1= {\rm d}\omega_2 = {\rm d}\omega_3= {\rm d}\omega _7=0, \\
{\rm d}\omega_4 = \omega_1 \wedge \omega_2 +\omega_3\wedge \omega_7,\\
{\rm d}\omega_5 = \omega_1 \wedge \omega_3 -\omega_2 \wedge \omega_7,\\
{\rm d}\omega_6 = \omega_2 \wedge \omega_3 +\omega_1 \wedge \omega_7,\\
\end{cases}
\end{gather}
which is of dimension $7$ sometimes called the Kaplan Lie algebra or the generalized Heisenberg algebra,
\begin{gather}\label{n81}
 \begin{cases}
 {\rm d}\omega_1= {\rm d}\omega_2 = {\rm d}\omega_3= {\rm d}\omega _7={\rm d}\omega _8=0, \\
{\rm d}\omega_4 = \omega_1 \wedge \omega_2 +\omega_3\wedge \omega_7,\\
{\rm d}\omega_5 = \omega_1 \wedge \omega_3 +\omega_2 \wedge \omega_8,\\
{\rm d}\omega_6 = \omega_2 \wedge \omega_3 +\omega_1 \wedge \omega_7, \end{cases}
\end{gather}
of dimension $8$,
 \begin{gather}\label{n91}
 \begin{cases}
 {\rm d}\omega_1= {\rm d}\omega_2 = {\rm d}\omega_3= {\rm d}\omega _7={\rm d}\omega _8={\rm d}\omega _9=0, \\
{\rm d}\omega_4 = \omega_1 \wedge \omega_2 +\omega_3\wedge \omega_7,\\
{\rm d}\omega_5 = \omega_1 \wedge \omega_3 +\omega_2 \wedge \omega_8,\\
{\rm d}\omega_6 = \omega_2 \wedge \omega_3 +\omega_1 \wedge \omega_9, \end{cases}
\end{gather}
of dimension $9$.

\item $\dim[\mm,\mm]=2$:
 \begin{gather}\label{n82}
 \begin{cases}
 {\rm d}\omega_1= {\rm d}\omega_2 = {\rm d}\omega_3= {\rm d}\omega _7={\rm d}\omega _8=0, \\
{\rm d}\omega_4 = \omega_1 \wedge \omega_2 +\omega_3\wedge \omega_7,\\
{\rm d}\omega_5 = \omega_1 \wedge \omega_3 +\omega_2 \wedge \omega_7,\\
{\rm d}\omega_6 = \omega_1 \wedge \omega_7 +\omega_2 \wedge \omega_8, \end{cases}
\end{gather}
which is of dimension $8$,
\begin{gather}\label{n92}
 \begin{cases}
 {\rm d}\omega_1= {\rm d}\omega_2 = {\rm d}\omega_3= {\rm d}\omega _7={\rm d}\omega _8={\rm d}\omega_9=0, \\
{\rm d}\omega_4 = \omega_1 \wedge \omega_2 +\omega_3\wedge \omega_7,\\
{\rm d}\omega_5 = \omega_1 \wedge \omega_3 +\omega_2 \wedge \omega_9,\\
{\rm d}\omega_6 = \omega_1 \wedge \omega_7 +\omega_2 \wedge \omega_8, \end{cases}
\end{gather}
which is of dimension $9$.

\item $\dim [\mm,\mm]=1$
\begin{gather}\label{n83}
 \begin{cases}
 {\rm d}\omega_1= {\rm d}\omega_2 = {\rm d}\omega_3= {\rm d}\omega _7={\rm d}\omega _8=0, \\
{\rm d}\omega_4 = \omega_1 \wedge \omega_2 +\omega_3\wedge \omega_7,\\
{\rm d}\omega_5 = \omega_1 \wedge \omega_8 +\omega_2 \wedge \omega_7,\\
{\rm d}\omega_6 = \omega_1 \wedge \omega_7 +\omega_3 \wedge \omega_8, \end{cases}
\end{gather}
which is of dimension $8$.
\item $\dim [\mm,\mm]=0$
\begin{gather}\label{n84}
 \begin{cases}
 {\rm d}\omega_1= {\rm d}\omega_2 = {\rm d}\omega_3= {\rm d}\omega _7={\rm d}\omega _8=0, \\
{\rm d}\omega_4 = \omega_1 \wedge \omega_7 +\omega_2\wedge \omega_8,\\
{\rm d}\omega_5 = \omega_1 \wedge \omega_8 +\omega_3 \wedge \omega_7,\\
{\rm d}\omega_6 = \omega_2 \wedge \omega_7 +\omega_1 \wedge \omega_8, \end{cases}
\end{gather}
which is of dimension $8$.
\end{enumerate}

Assume now that $\dim \mm=4$. In this case $\g/I$ is abelian or isomorphic to the solvable Lie algebra whose Maurer--Cartan equations are
\begin{gather*}
\begin{cases}
{\rm d}\omega_2={\rm d}\omega_4=0,\\
{\rm d}\omega_1=\omega_1 \wedge \omega_2+ \omega_3 \wedge \omega_4, \\
{\rm d}\omega_3=\omega_3 \wedge \omega_2+a \omega_1 \wedge \omega_4
\end{cases}
\end{gather*}
with $a \neq 0$. Let us assume that $\g/I$ is not abelian. Let $\{X_1,\dots,X_n\}$ be a basis of $\g$ such that $\{X_1,\dots,X_4\}$ is the basis of $\mm$ dual of $\{ \omega_1,\dots,\omega_4\}$ and $\{X_5,\dots,X_n\}$ a basis of $I$. Since $I$ is maximal, then $[X_1,I]$ and $[X_3,I]$ are not trivial. There exists a vector of~$I$, for example,~$X_5$ such that $[X_1,X_5] \neq 0$. Let us put $[X_1,X_5]=Y$ with $Y \in I$ and let be $\omega$ its dual form. Then
\begin{gather*}{\rm d}\omega= \omega_1 \wedge \omega_5 + \theta\end{gather*}
with $\omega_1 \wedge \omega_5 \wedge \theta \neq 0 $ and $ \omega_3 \wedge \omega_4 \wedge \theta=\omega_3 \wedge \omega_2 \wedge \theta=0$ if not there exists a linear form of class greater that $5$. This implies that they there exists $\omega_6$ independent with $\omega_5$
\begin{gather*}{\rm d}\omega= \omega_1 \wedge \omega_5 +b \omega_3\wedge \omega_6 \end{gather*}
with $b \neq 0$. Now the Jacobi conditions which are equivalent to ${\rm d}({\rm d}\omega)=0$ implies that we cannot have $\omega=\omega_5$ and $\omega=\omega_6$. Then we put $\omega=\omega_7$. This implies
\begin{gather*}{\rm d}\omega_7=\omega_1\wedge \omega_5+b_2 \omega_3 \wedge \omega_6, \qquad\!\! {\rm d}\omega_5=\omega_2\wedge \omega_5+b_3\omega_4 \wedge \omega_6, \qquad\!\! {\rm d}\omega_6=\omega_4\wedge \omega_5+b_4 \omega_2 \wedge \omega_6.\end{gather*}
We deduce that $\g$ is isomorphic to the Lie algebra whose Maurer--Cartan equations are
\begin{gather}\label{g9}
\begin{cases}
{\rm d}\omega_2={\rm d}\omega_4=0,\\
{\rm d}\omega_1=\omega_1 \wedge \omega_2+ \omega_3 \wedge \omega_4 , \\
{\rm d}\omega_3=\omega_3 \wedge \omega_2+a_1 \omega_1 \wedge \omega_4 , \\
 {\rm d}\omega_5=\omega_2\wedge \omega_5+a_2\omega_4 \wedge \omega_6, \\
 {\rm d}\omega_6=\omega_4\wedge \omega_5+a_3 \omega_2 \wedge \omega_6, \\
 {\rm d}\omega_7=\omega_1\wedge \omega_5+a_4 \omega_3 \wedge \omega_6
\end{cases}
\end{gather}
with $a_1a_2a_3a_4 \neq 0$.

If $\dim \mm \geq 5$, then $\dim A(\omega) =4$ or $0$ and the codimension of $I$ is greater than $n-4$. Then $\dim \mm\leq 4$.
\begin{Proposition}
Let $\g$ be a Lie algebra whose coadjoint orbit of any nonclosed linear form is of dimension $4$. Then
\begin{enumerate}\itemsep=0pt
 \item[$1)$] $\g$ is isomorphic to \eqref{n72}, \eqref{n81}, \eqref{n91}, \eqref{n82}, \eqref{n92}, \eqref{n83}, \eqref{n84}, \eqref{g9},
 \item[$2)$] or $\g$ admits a decomposition $\g=\mm \oplus I$ where~$I$ is a codimension~$2$ abelian ideal.
 \end{enumerate}
\end{Proposition}

\subsection[Classification when $\operatorname{codim} I=2$]{Classification when $\boldsymbol{\operatorname{codim} I=2}$}
These Lie algebras are described in Proposition \ref{c4}. It remains to classify them up to isomorphism. Let $\{T_1,T_2\}$ be a basis of $\mm$ and $\widetilde{\g}=\g/\K\{T_2\}$.

$\bullet$ $\dim \g =4$. Then $\dim \widetilde{\g}=3$ and it is isomorphic to one of the two algebras whose Lie brackets are given by
\begin{itemize}\itemsep=0pt
 \item $[T_1,X_1]=X_2$,
 \item or $[T_1,X_1]=X_1$, $[T_1,X_2]=aX_2$, $a \neq 0$.
\end{itemize}
In the first case, it is easy to see that we cannot find derivation of $\widetilde{\g}$ satisfying Proposition~\ref{c4}. In the second case the matrix of $f$ in the basis $\{T_1,X_1,X_2\}$ is
\begin{gather*}\begin{pmatrix}
 0 & 0 & 0 \\
 0 & b & c \\
 0 & d & e
\end{pmatrix}.
\end{gather*}
We have no solution if $a \neq 1$. If $a=1$ then $f$ satisfies
\begin{gather*}(e-b)^2+4cd <0.\end{gather*}
In particular for $b=e=0$ and $c=1$ we obtain
\begin{Proposition}
Any $4$-dimensional Lie algebra whose coadjoint orbits of nonclosed linear forms are of dimension $4$ is isomorphic to the Lie algebra $\g_4(\lambda)$ whose Maurer--Cartan equations are
\begin{gather*}
\begin{cases}
{\rm d}\alpha_1={\rm d}\alpha_2=0,\\
{\rm d}\omega_1=\alpha_1 \wedge \omega_1 +\alpha_2 \wedge \omega_2, \\
{\rm d}\omega_2= \alpha_1 \wedge \omega_2+ \lambda\alpha_2 \wedge \omega_1
\end{cases}
\end{gather*}
with $\lambda<0$.
\end{Proposition}
This Lie algebra corresponds to the Lie algebra $M_a^{10}$ in the classification of W.A.~de Graaf (see \cite{Gr, GozeEllipse}).

$\bullet$ $\dim \g=5$. Let us put $\widetilde{\g}=\K\{T_1\} \oplus I$. Let $h_1$ be the restriction of $\operatorname{ad} T_1$ to $I$. It is an endomorphism of~$I$ and since $\dim I=3$, it admits an eigenvalues~$\lambda$. Assume that $\lambda \neq 0$ and let~$X_1$ be an associated eigenvector. Then, since~$f$ is a derivation commuting with $\operatorname{ad} T_1$,
\begin{gather*}[T_1,f(X_1)]=f([T_1,X_1])=\lambda f(X_1).\end{gather*}
Then $f(X_1)$ is also an eigenvector associated with $\lambda$. By hypothesis $X_1$ and $f(X_1)$ are independent and $\lambda$ is a root of order $2$. Thus $h_1$ is semisimple. Let $\lambda_2$ be the third eigenvalue. If $X_3$ is an associated eigenvector, as above $f(X_3)$ is also an eigenvector and~$\lambda_2$ is a root of order $2$ except if $\lambda_2=\lambda_1$. Then
\begin{Lemma}If $\dim I$ is odd, then if the restriction $h_1$ of $\operatorname{ad} T_1$ to $I$ admits a nonzero eigenvalue, we have $h_1=\lambda \,{\rm Id}$.
\end{Lemma}
\begin{proof} We have proved this lemma for $\dim I =3$. By induction we find the general case.\end{proof}

We assume that $h_1=\lambda\,{\rm Id}$ with $\lambda \neq 0$. The derivation~$f$ of~$I$ is of rank~$3$ because~$f$ and~$h_1$ have the same rank by hypothesis. Since~$f$ is an endomorphism in a~$3$-dimensional space, it admits a~non-zero eigenvalue~$\mu$. Let $Y$ be an associated eigenvector, then
\begin{gather*}f(Y)=\mu Y, \qquad h_1(Y)=\lambda Y.\end{gather*}
This implies that there exists $Y$ such that $f(Y)$ and $h_1(Y)$ are not linearly independent. This is a contradiction. We deduce that $\lambda=0$.

As consequence, all eigenvalues of $h_1$ are null and $h_1$ is a nilpotent operator. In particular $\dim \operatorname{Im}(h_1) \leq 2$. If this rank is equal to $2$, the kernel is of dimension~$1$. Let $X_1$ be a generator of this kernel. Then $[T_1,X_1]=0$ this implies that $f(X_1)=0$ because $f$ and $h_1$ have the same image. We deduce that the subspace of $\g$ generated by $X_1$ is an abelian ideal and $\g$ is not indecomposable. Then $\dim \operatorname{Im}(h_1)=1$ and $\g$ is the $5$-dimensional Heisenberg algebra.

\begin{Proposition}Any $5$-dimensional Lie algebra whose coadjoint orbits of nonclosed linear forms are of dimension $4$ is isomorphic to the $5$-dimensional Heisenberg algebra whose Maurer--Cartan equations are
\begin{gather*}
\begin{cases}
 {\rm d}\omega_1={\rm d}\omega_2={\rm d}\omega_3={\rm d}\omega_4=0, \\
 {\rm d}\omega _5= \omega_1 \wedge \omega_2 +\omega_3 \wedge \omega_4.
\end{cases}
\end{gather*}
\end{Proposition}

$\bullet$ $\dim \g \geq 6$. {\it Solvable non-nilpotent case.} Since the Cartan class of any linear form on a~solvable Lie algebra is odd if and only if this Lie algebra is nilpotent, then if we assume that~$\g$ is solvable non-nilpotent, there exists a linear form of Cartan class equal to $4$. We assume also that $\g=\mm \oplus I$ with $\dim \mm =2$ and satisfying Proposition \ref{c4}. The determination of these Lie algebras is similar to (\ref{g9}) without the hypothesis $[X_1,I] \neq 0$ and $[X_3,I] \neq 0$. In this case $X_1$ and $X_3$ are also in $I$. We deduce immediately:
\begin{Proposition}Let $\g$ be a solvable non-nilpotent Lie algebra $\g=\mm \oplus I$ with $\dim \mm=2$, $I$ is an abelian ideal and whose dimensions of nontrivial coadjoint orbits are equal to~$4$. Then $\g$ is isomorphic to the following Lie algebra whose Maurer--Cartan equations are
\begin{gather*}%\label{s4}
\begin{cases}
{\rm d}\omega_2={\rm d}\omega_4=0,\\
{\rm d}\omega_{2l+2+3j}= {\rm d}\omega_{2l+3+3j}=0, \qquad j=0, \dots, s, \\
{\rm d}\omega_1=\omega_1 \wedge \omega_2+ \omega_3 \wedge \omega_4 , \\
{\rm d}\omega_3=\omega_3 \wedge \omega_2+a_1 \omega_1 \wedge \omega_4, \\
 {\rm d}\omega_5=\omega_2\wedge \omega_5+a_2\omega_4 \wedge \omega_6, \\
 {\rm d}\omega_6=\omega_2\wedge \omega_6+a_3 \omega_4 \wedge \omega_5, \\
 \cdots\cdots\cdots\cdots\cdots\cdots\cdots\cdots\cdots\cdots\cdots \\
 {\rm d}\omega_{2l-1}=\omega_2\wedge \omega_{2l-1}+a_{2l-4}\omega_4 \wedge \omega_{2l}, \\
 {\rm d}\omega_{2l}=\omega_2 \wedge \omega_{2l} +a_{2l-3}\omega_4\wedge \omega_{2l-1}, \\
 {\rm d}\omega_{2l+1+3j}=\omega_2\wedge \omega_{2l+2+3j}+\omega_4 \wedge \omega_{2l+3+3j}, \qquad j=0, \dots, s
 \end{cases}
\end{gather*}
with $a_1 \cdots a_{2l-3} \neq 0$.
\end{Proposition}

$\bullet$ $\dim \g \geq 6$. {\it Nilpotent case.} Let us describe nilpotent algebras of type $\mm \oplus I$ where $I$ is a~maximal abelian ideal with $\dim A({\rm d}\omega)= n-4$ or $n$ for any $\omega \in \g^*$. Let us recall also that the Cartan class of any nontrivial linear form is odd then here equal to~$5$. In the previous examples, we have seen that for the $5$-dimensional case, we have obtained only the Heisenberg algebra. Before to study the general case, we begin by a description of an interesting example. Let us consider the following nilpotent Lie algebra, denoted by $\h(p,2)$ given by
\begin{gather*}
\begin{cases}
 {\rm d}\omega_1=\alpha_1 \wedge \beta_1+ \alpha_2 \wedge \beta_2, \\
 {\rm d}\omega_2=\alpha_1 \wedge \beta_3+ \alpha_2 \wedge \beta_4, \\
 \cdots\cdots\cdots\cdots\cdots\cdots\cdots\cdots\cdots \\
 {\rm d}\omega_p=\alpha_1 \wedge \beta_{2p-1}+ \alpha_2 \wedge \beta_{2p}, \\
 {\rm d}\alpha_1={\rm d}\alpha_2={\rm d}\beta_i=0,\qquad i=1,\dots, 2p.
\end{cases}
\end{gather*}
This Lie algebra is nilpotent of dimension $3p+2$ and it has been introduced in~\cite{GH} in the study of Pfaffian system of rank greater than $1$ and of maximal class.
\begin{Proposition}For any nonclosed linear form on $\h(p,2)$, the dimension of the coadjoint orbit is equal to $4$.
\end{Proposition}

To study the general case, we shall use the notion of characteristic sequence which is an invariant up to isomorphism of nilpotent Lie algebras (see for example \cite{RBreadth} for a presentation of this notion). For any $X \in \g$, let $c(X)$ be the ordered sequence, for the lexicographic order, of the dimensions of the Jordan blocks of the nilpotent operator $\operatorname{ad} X$. The characteristic sequence of $\g$ is the invariant, up to isomorphism,
\begin{gather*}c(\g)=\max \big\{c(X), \, X \in \g-\mathcal{C}^1(g)\big\}.\end{gather*}
In particular, if $c(\g)=(c_1,c_2,\dots,1)$, then $\g$ is $c_1$-step nilpotent. For example, we have $c(\h(p,2))=(c_1=2,\dots,c_p=2,1,\dots,1)$. A vector $X \in \g$ such that $c(X)=c(\g)$ is called a characteristic vector of $\g$.

\begin{Theorem}
Let $\g$ be a nilpotent Lie algebra such that the dimension of the coadjoint orbit of a nonclosed form is $4$ and admitting the decomposition $\g=\mm \oplus I$ where $I$ is an abelian ideal of codimension~$2$. Then~$\mm$ admits a basis $\{T_1,T_2\}$ of characteristic vector of~$\g$ with the same characteristic sequence and $\operatorname{Im} (\operatorname{ad} T_1)=\operatorname{Im} (\operatorname{ad} T_2)$.
\end{Theorem}

\begin{proof} Let $T$ be a nonnull vector of $\mm$ such that $\widetilde{\g}=\g/\K\{T\}$ is a nilpotent Lie algebra given in~(\ref{Lc}). Then $\widetilde{\g}=\mm_1 \oplus I$ and if $T_1 \in \mm_1$, $T_1 \neq 0$, then~$T_1$ is a characteristic vector of~$\widetilde{\g}$. Then~$T_1$ can be considered as a~characteristic vector of $\g$. Let be $T_2 \in \mm$ such that $\operatorname{ad} T_1$ and $\operatorname{ad} T_2$ have the same image in~$I$. Then $T_2$ is also a characteristic vector with same characteristic sequence, if not $c(T_1)$ will be not maximal.
\end{proof}

Let us consider a Jordan basis $\big\{X_1^1,\dots,X_1^{c_1},X_2^1,\dots,X_2^{c_2}, \dots X_k^{c_k},T_2,T_1\big\}$ of $\operatorname{ad} T_1$ correspon\-ding to the characteristic sequence $c(\g)=(c_1,\dots,c_k,1,1)$. In the dual basis \begin{gather*}\big\{\omega_1^1,\dots, \omega_1^{c_1},\omega_2^1,\dots,\omega_2^{c_2}, \dots \omega_k^{c_k},\alpha_2,\alpha_1\big\}\end{gather*} we have
\begin{gather*}{\rm d}\omega_s^{j}= \alpha_1 \wedge \omega_s^{j-1} + \alpha_2 \wedge \beta_s^j\end{gather*}
for any $s=1,\dots,k$ and $j=1,\dots,c_s$ where $\beta_i^{s}$ satisfies $\beta_i^{s} \wedge \omega_i^{s-1} \neq 0$. The Jacobi conditions imply that $\beta_s^{1}, \dots,\beta_s^{c_s}$ are the dual basis of a Jordan block. This implies

\begin{Lemma}
If $c(\g)=(c_1,\dots,c_k,1,1)$ is the characteristic sequence of $\operatorname{ad} T_1$, then for any $c_s \in c(\g)$, $c_s \neq 1$, then $c_s-1$ is also in $c(\g)$.
\end{Lemma}

Thus, if $c(\g)$ is a strictly decreasing sequence, that is if $c_s > c_{s-1}$ for $c_s \geq 2$, we shall have
\begin{gather*}c(\g)=(c_1,c_1-1,c_1-2,\dots,2,1,1) , \qquad {\rm or } \qquad c(\g)=(c_1,c_1-1,c_1-2,\dots,2,1,1,1).\end{gather*}
In all the other cases, we shall have
\begin{gather*}c(\g)=(c_1,\dots,c_1,c_1-1,\dots,c_1-1,c_1-2,\dots, 1).\end{gather*}

\looseness=-1 Let us describe the nilpotent Lie algebras whose the characteristic sequence is strictly decreasing, the other cases can be deduced. Assume that $c_1=l \geq 3$. Then $\g$ is isomorphic to
\begin{gather}
 \big[T_1,X^i_1\big]=X^{i+1}_1, \qquad i=1,\dots,l-1, \qquad \big[T_2, X^i_1\big]=0 , \qquad i=1,\dots,l-1 ,\nonumber\\
 \big[ T_1,X_2^i\big]=X_2^{i+1}, \qquad i=1,\dots,l-2, \qquad \big[ T_2 X^i_2\big]=X_1^{i+1}, \qquad i=1,\dots,l-1, \nonumber\\
 \cdots\cdots\cdots\cdots\cdots\cdots\cdots\cdots\cdots\cdots\cdots\cdots\cdots\cdots\cdots\cdots\cdots\cdots \cdots\nonumber\\
 \big[ T_1,X_{l-1}^1\big]=X_{l-1}^{2}, \qquad \big[ T_2, X_{l-1}^i\big]=X_{l-2}^{i+1} , \qquad i=1,2,\nonumber \\
 \big[ T_1,X_{l}^1\big]=0, \qquad \big[ T_2, X_{l}^1\big]=X_{l-1}^{2}, \label{n1}
\end{gather}
 when $c(\g)=(l,l-1,\dots,2,1,1,1)$ or
\begin{gather}
 \big[T_1,X^i_1\big]=X^{i+1}_1, \qquad i=1,\dots,l-1,\nonumber\\ \big[T_2, X^l_1\big]= X_{l-1}^{2}, \qquad \big[T_2, X^i_1\big]=0 , \qquad i=1,\dots,l-1, \nonumber\\
 \big[ T_1,X_2^i\big]=X_2^{i+1}, \qquad i=1,\dots,l-2, \qquad \big[ T_2 X^i_2\big]=X_1^{i+1}, \qquad i=1,\dots,l-1, \nonumber\\
 \cdots \cdots\cdots\cdots\cdots\cdots\cdots\cdots\cdots\cdots\cdots\cdots\cdots\cdots\cdots\cdots\cdots\cdots\cdots\nonumber\\
 \big[ T_1,X_{l-1}^1\big]=X_{l-1}^{2},\qquad \big[ T_2, X_{l-1}^i\big]=X_{l-2}^{i+1}, \qquad i=1,2, \label{n2}
\end{gather}
 when $c(\g)=(l,l-1,\dots,2,1,1)$. In particular we deduce
 \begin{Proposition} Let $\g$ be a $n$-dimensional nilpotent Lie algebra whose nontrivial coadjoint orbits are of dimension~$4$. Then if $c (\g)=(c_1,\dots,c_k,1)$ is its characteristic sequence, then
 \begin{gather*}c_1 \leq \frac{\sqrt{8n-7}-1}{2}.\end{gather*}
 \end{Proposition}
 In fact, for the Lie algebras (\ref{n1}) or (\ref{n2}) we have $2n= \frac{l(l+1)}{2}+1$ or $2n= \frac{l(l+1)}{2}+2$. Other nilpotent Lie algebras satisfying this hypothesis on the dimension of coadjoint orbits have charac\-teristic sequences $(c_1,\dots,c_k,1)$ with $c_1 \geq c_2\dots \geq k \geq 1$, the inequalities being here not strict.

If $c_1=2$, the Lie algebra is $2$-step nilpotent and the characteristic sequence of $\g$ is of type $c(\g)=(\underbrace{2,\dots,2}_l,\underbrace{1,\dots,1}_{l-s+2})$ where $s \leq l$.
In fact, $\g$ is an extension by a derivation of $\widetilde{\g}=\mm_1 \oplus I$ which is equal to $0$ on $\mm_1$. Let us note that the characteristic sequence of the Kaplan algebra~(\ref{n72}) is~$(2,2,2,1)$, but this Lie algebra is a particular case which do not corresponds to the previous decomposition. In general, the Maurer--Cartan equations will be done by
\begin{gather*}
\begin{cases}
 {\rm d}\omega^2_1 = \alpha_1 \wedge \omega_1^1+\alpha_2 \wedge \beta_1, \\
 {\rm d}\omega^2_2 = \alpha_1 \wedge \omega^1_2+\alpha_2 \wedge \beta_2, \\
 \cdots \cdots \cdots \cdots \cdots \cdots \cdots \cdots \\
 {\rm d}\omega^2_l = \alpha_1 \wedge \omega^1_l+\alpha_2 \wedge \beta_l
\end{cases}
\end{gather*}
with $d \alpha_1 ={\rm d}\alpha_2={\rm d}\omega_1^i={\rm d}\beta_i =0$, $i=1,\dots,l$. and $\beta_{s+1} \wedge \dots \wedge \beta_l \wedge \omega_1^1 \wedge \dots \wedge \omega_1^l \neq 0$, and for any $i=1,\dots,s$, $\beta_i \in \R\big\{\omega_1^1,\dots,\omega_1^l\big\}$ with
\begin{gather*}\omega_i^1 \wedge \beta_j - \omega_j^1 \wedge \beta_i \neq 0 \end{gather*}
for any $i,j \in \{1, \dots,s\}$.

\section{Lie algebras with coadjoint orbits of maximal dimension}

In this section we are interested by $n$-dimensional Lie algebra admitting orbits of dimension~$n$ if~$n$ is even or dimension $n-1$ is~$n$ is odd. This is equivalent to consider Frobenius Lie algebras in the first case and contact Lie algebras in the second case.

\subsection[$(2p+1)$-dimensional Lie algebras with a $2p$-dimensional coadjoint orbit]{$\boldsymbol{(2p+1)}$-dimensional Lie algebras with a $\boldsymbol{2p}$-dimensional coadjoint orbit}

Let $\h_{2p+1}$ be the $(2p+1)$-dimensional Heisenberg algebra. There exists a basis $\{X_1,\dots,X_{2p+1}\}$ such that the Lie brackets of $\h_{2p+1}$ are $[X_1,X_2]=\dots =[X_{2p-1},X_{2p}]=X_{2p+1}$, and $[X_i,X_j]=0$ for $i<j$ and $ (i,j)\notin \{(1,2),\dots,(2p-1,2p) \}$. If $\{\omega_1,\dots,\omega_{2p+1}\}$ denotes the dual basis of $\h_{2p+1}^*$, the Maurer--Cartan equations writes
\begin{gather*}{\rm d}\omega_{2p+1}=-\omega_1 \wedge \omega_2 - \dots -\omega_{2p-1} \wedge \omega_{2p}, \qquad {\rm d}\omega_i=0 ,\qquad i=1, \dots , 2p.\end{gather*}
Then $\omega_{2p+1}$ is a contact form on $\h_{2p+1}$ and the coadjoint orbit $\mathcal{O}(\omega_{2p+1})$ is of maximal dimen\-sion~$2p$. Let us note that, in $\h_{2p+1}$ all the orbits are of dimension $2p$ or are singular. In the following, we will denote by $\mu_0$ the Lie bracket of $\h_{2p+1}$.

\begin{Definition}A formal quadratic deformation $\g_t$ of $\h_{2p+1}$ is a $(2p+1)$-dimensional Lie algebra whose Lie bracket $\mu_t$ is given by
\begin{gather*}\mu_t(X,Y)=\mu_0(X,Y)+t\varphi_1(X,Y)+t^2\varphi_2(X,Y), \end{gather*}
where the maps $\varphi_i$ are bilinear on $\h_{2p+1}$ with values in $\h_{2p+1}$ and satisfying
\begin{alignat*}{3}
& \delta_{\mu_0} \varphi_1 =0, \qquad && \varphi_1\circ \varphi _1+ \delta_{\mu_0} \varphi_2=0, &\nonumber\\
& \varphi_2\circ \varphi _2=0, \qquad && \varphi_1\circ \varphi _2+\varphi_2\circ \varphi _1=0.&%\label{quadra}
\end{alignat*}
\end{Definition}
In this definition $\delta_{\mu} $ denotes the coboundary operator of the Chevalley--Eilenberg cohomology of a Lie algebra $\g$ whose Lie bracket is $\mu$ with values in~$\g$, and if $\varphi$ and $\psi$ are bilinear maps, then $\varphi \circ\psi$ is the trilinear map given by
\begin{gather*}\varphi \circ\psi(X,Y,Z)=\varphi(\psi(X,Y),Z)+ \varphi(\psi(Y,Z),X)+\varphi(\psi(Z,X),Y).\end{gather*} In particular $\varphi\circ\varphi=0$ is equivalent to Jacobi identity and~$\varphi$, in this case, is a Lie bracket and the coboundary operator writes
\begin{gather*}\delta_\mu \varphi= \mu \circ \varphi + \varphi \circ \mu.\end{gather*}

\begin{Theorem}[\cite{GR-DGA}]Any $(2p+1)$-dimensional contact Lie algebra $\g$ is isomorphic to a quadratic formal deformation of $\h_{2p+1}$.
\end{Theorem}
Then its Lie bracket $\mu$ writes $\mu=\mu_0+t\varphi_1+t^2\varphi_2$ and the bilinear maps $\varphi_1$ and $\varphi_2$ have the following expressions in the basis $\{X_1,\dots,X_{2p+1}\}$:
\begin{gather*}
\varphi_1(X_{2k-1},X_{2k})=\sum_{s=1}^{2p}C_{2k-1,2k}^sX_s, \qquad k=1,\dots,p,\nonumber\\
\varphi_1(X_l,X_r)=\sum_{s=1}^{2p}C_{l,r}^sX_s, \qquad 1\leq l<r\leq 2p,\qquad (l,r)\neq (2k-1,2k),%\label{phi1}
\end{gather*}
and
\begin{gather}\begin{split}&
\varphi_2(X_l,X_r)=0, \qquad l,r=1,\dots,2p,\nonumber\\
& \varphi_2(X_l,X_{2p+1})=\sum_{s=1}^{2p}C_{l,{2p+1}}^sX_s, \qquad l=1,\dots,2p,%\label{phi2}
\end{split}
\end{gather}
and the nondefined values are equal to $0$. Since the center of any deformation of~$\h_{2p+1}$ is of dimension less than or equal to the dimension of $\h_{2p+1}$, we deduce:

\begin{Corollary}[\cite{GozeCras1}]\label{center}The center $Z(\g)$ of a contact Lie algebra $\g$ is of dimension less than or equal to~$1$.
\end{Corollary}

\subsubsection{Case of nilpotent Lie algebras}
If $\g$ is a contact nilpotent Lie algebra, its center is of dimension $1$. In the given basis, this center is $\R\{X_{2p+1}\}$. This implies that $\varphi_2=0$.
\begin{Proposition}Any $(2p+1)$-dimensional contact nilpotent Lie algebra is isomorphic to a~li\-near deformation $\mu_t=\mu_0+t\varphi_1$ of the Heisenberg algebra~$\h_{2p+1}$.
\end{Proposition}

As consequence, we have
\begin{Corollary}Any $(2p+1)$-dimensional contact nilpotent Lie algebra is isomorphic to a central extension of a $2p$-dimensional symplectic Lie algebra by its symplectic form.
\end{Corollary}

\begin{proof} Let $\frak{t}$ be the $2p$-dimensional vector space generated by $\{X_1,\dots,X_{2p}\}$. The restriction to~$\frak{t}$ of the $2$-cocycle $\varphi_1$ is with values in $\frak{t}$. Since $\varphi_1 \circ \varphi_1=0$, it defines on~$\frak{t}$ a structure of $2p$-dimensional Lie algebra. If $\{\omega_1,\dots,\omega_{2p+1}\}$ is the dual basis of the given classical basis of $\h_{2p+1}$, then $\theta=\omega_1\wedge\omega_2+\dots +\omega_{2p-1}\wedge\omega_{2p}$ is a $2$-form on $\frak{t}$. We denote by ${\rm d}_{\varphi_1}$ the differential operator on the Lie algebra $(\frak{t},\varphi_1)$, that is ${\rm d}_{\varphi_1}\omega(X,Y)=-\omega(\varphi_1(X,Y)$ for all $X,Y \in \frak{t}$ and $\omega \in \frak{t}^*$. Since~$\mu_0$ is the Heisenberg Lie algebra multiplication, the condition $\D\varphi_1=0$ is equivalent to ${\rm d}_{\varphi_1}(\theta)=0$. It implies that $\theta $ is a closed $2$-form on~$\frak{t}$ and~$\g$ is a central extension of~$\frak{t}$ by~$\theta$.
\end{proof}

We deduce:

\begin{Theorem}[\cite{GR-DGA}]\label{contactnilpotent} Let $\g$ be a $(2p+1)$-dimensional $k$-step nilpotent Lie algebra. Then there exists on $\g$ a coadjoint orbit of dimension $2p$ if and only if $\g$ is a central extension of a $2p$-dimensional $(k-1)$-step nilpotent symplectic Lie algebra $\frak{t}$, the extension being defined by the $2$-cocycle given by the symplectic form.
\end{Theorem}

Since the classification of nilpotent Lie algebras is known up the dimension $7$, the previous result permits to establish the classification of contact nilpotent Lie algebras of dimension~$3$,~$5$ and~$7$ using the classification in dimension $2$,~$4$ and $6$~\cite{BG}. For example, the classification of $5$-dimensional nilpotent Lie algebras with an orbit of dimension $4$ is the following:
\begin{itemize}\itemsep=0pt
\item $\g$ is $4$-step nilpotent:

$\frak{n}^1_5$: $[ X_1,X_i]=X_{i+1}$, $ i=2,3,4$, $[ X_2,X_3 ]=X_5$.

\item $\g$ is $3$-step nilpotent:

$\frak{n}^3_5$: $[ X_1,X_i]=X_{i+1}$, $i=2,3$, $[ X_2,X_5]=X_4$.

\item $\g$ is $2$-step nilpotent:

$\frak{n}^6_5$: $[ X_1,X_2]=X_3$, $[ X_4,X_5]=X_3$.
\end{itemize}

We shall now study a contact structure in respect of the characteristic sequence of a nilpotent Lie algebra. For any $X \in \g$, let $c(X)$ be the ordered sequence, for the lexicographic order, of the dimensions of the Jordan blocks of the nilpotent operator $\operatorname{ad} X$. The characteristic sequence of~$\g$ is the invariant, up to isomorphism,
\begin{gather*}c(\g)=\max \big\{c(X), \, X \in \g-\mathcal{C}^1(\g)\big\}.\end{gather*}
Then $c(\g)$ is a sequence of type $(c_1,c_2,\dots,c_k,1)$ with $c_1 \geq c_2 \geq \cdots \geq c_k \geq 1$ and $c_1+c_2+\dots+c_k+1= n =\dim \g$. For example, \begin{enumerate}\itemsep=0pt
 \item[1)] $c(\g)=(1,1,\dots,1)$ if and only if $\g$ is abelian,
 \item[2)] $c(\g)=(2,1,\dots,1)$ if and only if $\g$ is a direct product of an Heisenberg Lie algebra by an abelian ideal,
 \item[3)] if $\g$ is $2$-step nilpotent then there exists $p$ and $q$ such that $c(\g)=(2,2,\dots,2,1,\dots,1)$ with $2p+q=n$, that is $p$ is the occurrence of $2$ in the characteristic sequence and $q$ the occurrence of~$1$,
 \item[4)] $\g$ is filiform if and only if $c(\g)=(n-1,1)$.
 \end{enumerate}
This invariant was introduced in~\cite{AG} in order to classify $7$-dimensional nilpotent Lie algebras. A~link between the notions of breath of nilpotent Lie algebra introduced in~\cite{KMS} and characteristic sequence is developed in~\cite{RBreadth}. If $c(\g)=(c_1,c_2,\dots,c_k,1)$ is the characteristic sequence of~$\g$ then~$\g$ is $c_1$-step nilpotent.

 Assume now that $c_1=c_2=\dots =c_l$ and $c_{l+1} < c_l$. Then the dimension of the center of~$\g$ is greater than~$l$ because in each Jordan blocks corresponding to $c_1,\dots,c_l$, the last vector is in~$\mathcal{C}^{c_1}(\g)$ which is contained in the center of~$\g$. We deduce

 \begin{Proposition} Let $\g$ be a contact nilpotent Lie algebra. Then its characteristic sequence is of type $c(\g)=(c_1,c_2,\dots,c_k,1)$ with $c_2 \neq c_1$.
 \end{Proposition}

\begin{Example}\quad
 \begin{enumerate}\itemsep=0pt
\item If $\g$ is $2$-step nilpotent, then $c(\g)=(2,\dots,2,1,\dots,1)$ and if $\g$ is a contact Lie algebra, we have $c(\g)=(2,1,\dots,1)$. We find again the results given in \cite{GR-DGA} which precise that any $2$-step nilpotent $(2p+1)$-dimensional contact Lie algebra is isomorphic to the Heisenberg algebra $\frak{h}_{2p+1}$.
 \item If $\g$ is $3$-step nilpotent, then $c(\g)=(3,2,\dots,2,1,\dots,1)$ or $c(\g)=(3,1,\dots,1)$. In case of dimension $7$, this gives $c(\g)=(3,2,1,1)$ and $c(\g)=(3,1,1,1,1)$. For each one of characteristic sequences there are contact Lie algebras:
 \begin{enumerate}\itemsep=0pt
 \item The Lie algebra given by
 \begin{gather*}[X_1,X_2]=X_3,\qquad [X_1,X_3]=[X_2,X_5]=[X_6,X_7]=X_4,\end{gather*}
 is a $7$-dimensional contact Lie algebra of characteristic $(3,1,1,1,1)$
 \item The Lie algebras given
 \begin{gather*}[X_1,X_i] = X_{i+1}, \qquad i = 2,3,5, \qquad [X_2,X_5] = X_7, \qquad [X_2,X_7] = X_4, \\ [X_5,X_6] = X_4,\qquad [X_5,X_7] = \alpha X_4\end{gather*}
 with $\alpha \neq 0$ are contact Lie algebras of characteristic $(3,2,1,1)$.
 \end{enumerate}
 \item If $\g$ is $4$-step nilpotent, then $c(\g)=(4,3,\dots,1)$. For example, the $9$-dimensional Lie algebra given by
 \begin{gather*}
 [ X_1,X_i ] =X_{i+1},\qquad i=2,3,4,6,7, \\
 [ X_6,X_9]=X_3, \qquad [ X_7,X_9 ]=X_4,\qquad [ X_8,X_9]=X_5, \\
 [ X_2,X_6]=(1+\alpha)X_4, \qquad [ X_3,X_6]=X_5,\qquad [ X_2,X_7]=\alpha X_5
\end{gather*}
 with $\alpha \neq 0$ is a contact Lie algebra with $c(\g)=(4,3,1,1)$.
 \end{enumerate}
Let us note also, that in \cite{Rfili}, we construct the contact nilpotent filiform Lie algebras, that is with characteristic sequence equal to $(2p,1)$.
\end{Example}

\subsubsection{The non-nilpotent case}
It is equivalent to consider Lie algebras with a contact form defined by a quadratic nonlinear deformations of the Heisenberg algebra. We refer to~\cite{GR-DGA} to the description of this class of Lie algebras.

An interesting particular case consists to determine all the $(2p+1)$-dimensional Lie algebras ($p \neq 0$), such that all the coadjoint orbits of nontrivial elements are of dimension $2p$.

\begin{Lemma}[\cite{GozeCras1}] Let $\g$ a simple Lie algebra of rank $r$ and dimension $n$. Then any nontrivial linear form $\alpha$ on $\g$ satisfies
\begin{gather*}\operatorname{cl}(\alpha) \leq n-r+1.\end{gather*}
Moreover, if $\g$ is of classical type, we have
\begin{gather*}\operatorname{cl}(\alpha) \geq 2r.\end{gather*}
\end{Lemma}

In particular, a simple Lie algebra admits a contact form if its rank is equal to $1$ and this Lie algebra is isomorphic to $\mathfrak{sl}(2,\R)$ or $\mathfrak{so}(3)$. From Proposition~\ref{simple} these algebras are the only Lie algebras of dimension~$2p$ or $2p+1$ whose orbits $\mathcal{O}(\alpha)$ with $\alpha \neq 0$ are of dimension~$2p$.

\subsection[$2p$-dimensional Lie algebras with a $2p$-dimensional coadjoint orbit]{$\boldsymbol{2p}$-dimensional Lie algebras with a $\boldsymbol{2p}$-dimensional coadjoint orbit}

Such Lie algebra is Frobenius (see~\cite{Ooms}). Since the Cartan class of a linear form on a nilpotent Lie algebra is always odd, this Lie algebra is not nilpotent. In the contact case, we have seen that any contact Lie algebra is a deformation of the Heisenberg algebra. On other words, any contact Lie algebra can be contracted on the Heisenberg algebra. In the Frobenius case, we have a similar but more complicated situation. We have to determine an irreducible family of Frobenius Lie algebras with the property that any Frobenius Lie algebra can be contracted on a Lie algebra of this family.

In a first step, we recall the notion of contraction of Lie algebras. Let $\g_0$ be a $n$-dimensional Lie algebra whose Lie bracket is denoted by $\mu_0$. We consider $ \{ f_{t} \} _{t \in ]0,1]} $ a sequence of isomorphisms in $\mathbb{K}^{n}$ with $\mathbb{K}=\R$ or $\C$. Any Lie bracket
\begin{gather*}
\mu_{t}=f_t^{-1}\circ\mu_{0} ( f_{t}\times f_{t} )
\end{gather*}
corresponds to a Lie algebra $\g_t$ which is isomorphic to $\g_0$. If the limit $\lim\limits_{t\rightarrow 0}\mu_{t}$ exists (this limit is computed in the finite-dimensional vector space of bilinear maps in~$\K^n$), it defines a Lie bracket~$\mu$ of a~$n$-dimensional Lie algebra $\g$ called a contraction of $\g_{0}$.

\begin{Remark} Let $\frak{L}^{n}$ be the variety of Lie algebra laws over $\mathbb{C}^{n}$ provided with its Zariski topology. The algebraic structure of this variety is defined by the Jacobi polynomial equations on the structure constants. The linear group ${\rm GL} ( n,\mathbb{C} ) $ acts on~$\mathbb{C}^{n}$ by changes of basis. A~Lie algebra~$\frak{g}$ is contracted to the $\g_0$ if the Lie bracket of~$\g$ is in the closure of the orbit of the Lie bracket of $\g_ 0$ by the group action (for more details see \cite{Bu, GozeEllipse}).
\end{Remark}

\subsubsection{Classification of complex Frobenius Lie algebras up to a contraction}
Let $\g$ be a $2p$-dimensional Frobenius Lie algebra. There exists a basis $\{X_1, \dots, X_{2p}\}$ of $\g$ such that the dual basis $\{\omega_1,\dots,\omega_{2p}\}$ is adapted to the Frobenius structure, that is, \begin{gather*}
{\rm d} \omega_1=\omega_1\wedge \omega_2+\omega_3\wedge \omega_4+\dots+\omega_{2p-1}\wedge \omega_{2p}.
\end{gather*} In the following, we define Lie algebras, not with its brackets, but with its Maurer--Cartan equations. We assume here that $\K=\C$.

\begin{Theorem}\label{frob} Let $\g_{a_1,\dots,a_{p-1}}$, $ a_i \in \C$ be the Lie algebras defined by
\begin{gather*}
{\rm d}\omega_{1}=\omega_{1}\wedge\omega_{2}+\sum_{k=1}^{p-1}\omega_{2k+1}\wedge\omega_{2k+2},\qquad {\rm d}\omega_{2}=0,\\
{\rm d}\omega_{2k+1}=a_{k}\omega_{2}\wedge\omega_{2k+1},\qquad 1\leq k\leq p-1,\\
{\rm d}\omega_{2k+2}=- ( 1+a_{k} ) \omega_{2}\wedge\omega_{2k+2},\qquad 1\leq k\leq p-1.
\end{gather*}
Then any $2p$-dimensional Frobenius Lie algebra can be contracted in an element of the family $\mathcal{F}=\{\g_{a_1,\dots,a_{p-1}}\}_{a_i \in \C}$. Moreover, any element of $\mathcal{F}$ cannot be contracted in another element of this family.
\end{Theorem}

\begin{proof} See \cite{GozeCras2} and \cite{GR-DGA}.\end{proof}

\begin{Remark}
The notion of principal element of a Frobenius Lie algebra is defined in~\cite{Ooms}. In the basis $\{X_1,\dots,X_{2p}\}$ for which ${\rm d}\omega_{1}=\omega_{1}\wedge\omega_{2}+\sum\limits_{k=1}^{p-1}\omega_{2k+1}\wedge\omega_{2k+2}$, the principal element is~$X_2$.
\end{Remark}

\begin{Proposition}
The parameters $\{a_1,\dots,a_{p-1}\}$ which are the invariants of Frobenius Lie algebras up to contraction are the eigenvalues of the principal element of $\g_{a_1,\dots,a_{p-1}}$.
\end{Proposition}

\subsubsection{Classification of real Frobenius Lie algebras up to contraction}

We have seen that, in the complex case, the classification up to contraction of $2p$-dimensional Lie algebras is in correspondence with the reduced matrix of the principal element. We deduce directly the classification in the real case:

\begin{Theorem}\label{frobreal}Let $\g_{a_1,\dots,a_{s},b_1,\dots,b_{2p-s-1}}$, $ a_i,b_j \in \R$ be the $2p$-dimensional Lie algebras given by
\begin{gather*}
[X_1,X_2]=[X_{2k-1},X_{2k}]=X_1, \qquad k=2,\dots,p, \\
[X_2,X_{4k-1}] =a_kX_{4k-1}+b_kX_{4k+1}, \qquad [ X_2,X_{4k}] =(-1-a_k)X_{4k}-b_kX_{4k+2}, \qquad k \leq s,\\
[ X_2,X_{4k+1}] =-b_kX_{4k-1}+a_kX_{4k+1}, \\ [ X_2,X_{4k+2}] =b_kX_{4k}+(-1-a_k)b_kX_{4k+2}, \qquad k \leq s,\\
[X_2,X_{4s+2k-1}] =-\tfrac{1}{2}X_{4s+2k-1}+b_kX_{4k+2k}, \qquad 2 \leq k \leq p-2s,\\
[ X_2,X_{4s+2k}] =-b_kX_{4s+2k-1}-\tfrac{1}{2}X_{4s+2k}, \qquad 2 \leq k \leq p-2s.
\end{gather*}
Then any $2p$-dimensional Frobenius real Lie algebra is contracted in an element of the family $\mathcal{F}=\{\g_{a_1,\dots,a_{s},b_1,\dots,b_{2p-s-1}}\}_{a_i,b_j \in \R}$. Moreover, any element of $\mathcal{F}$ cannot be contracted in other different element of this family.
\end{Theorem}

\begin{Corollary}Any $2p$-dimensional Lie algebras described in~\eqref{frob} in the complex case and in~\eqref{frobreal} in the real case have an open coadjoint orbit of dimension~$2p$. Moreover any $2p$-dimensional Lie algebras with an open coadjoint orbit of dimension~$2p$ is a deformation of a~Lie algebra of these families.
\end{Corollary}

\subsection*{Acknowledgements}

The authors would like to thank the referees for their kind advices and useful comments to improve the paper.

\pdfbookmark[1]{References}{ref}
\LastPageEnding


\begin{thebibliography}{99}
\footnotesize\itemsep=-0.9pt

\bibitem{Adimi}
Adimi H., Makhlouf A., Index of graded filiform and quasi filiform {L}ie
 algebras, \href{https://doi.org/10.2298/FIL1303467A}{\textit{Filomat}} \textbf{27} (2013), 467--483, \href{https://arxiv.org/abs/1212.1650}{arXiv:1212.1650}.

\bibitem{AG}
Ancoch\'{e}a-Berm\'{u}dez J.M., Goze M., Classification des alg\`ebres de {L}ie
 nilpotentes complexes de dimension~{$7$}, \href{https://doi.org/10.1007/BF01191272}{\textit{Arch. Math. (Basel)}}
 \textbf{52} (1989), 175--185.

\bibitem{Arnal}
Arnal D., Cahen M., Ludwig J., Lie groups whose coadjoint orbits are of
 dimension smaller or equal to two, \href{https://doi.org/10.1007/BF00739806}{\textit{Lett. Math. Phys.}} \textbf{33}
 (1995), 183--186.

\bibitem{AW}
Awane A., Goze M., Pfaffian systems, {$k$}-symplectic systems, \href{https://doi.org/10.1007/978-94-015-9526-1}{Kluwer Academic
 Publishers}, Dordrecht, 2000.

\bibitem{Beltita}
Belti\c{t}\u{a} D., Cahen B., Contractions of {L}ie algebras with 2-dimensional
 generic coadjoint orbits, \href{https://doi.org/10.1016/j.laa.2014.10.005}{\textit{Linear Algebra Appl.}} \textbf{466} (2015),
 41--63, \href{https://arxiv.org/abs/1401.3272}{arXiv:1401.3272}.

\bibitem{Beltita2}
Belti\c{t}\u{a} I., Belti\c{t}\u{a} D., Coadjoint orbits of stepwise square
 integrable representations, \href{https://doi.org/10.1090/proc/12761}{\textit{Proc. Amer. Math. Soc.}} \textbf{144}
 (2016), 1343--1350, \href{https://arxiv.org/abs/1408.1857}{arXiv:1408.1857}.

\bibitem{Be2}
Belti\c{t}\u{a} I., Belti\c{t}\u{a} D., On the isomorphism problem for
 $C^*$-algebras of nilpotent Lie groups, \href{https://arxiv.org/abs/1804.05562}{arXiv:1804.05562}.

\bibitem{Bu}
Burde D., Degenerations of 7-dimensional nilpotent {L}ie algebras,
 \href{https://doi.org/10.1081/AGB-200053956}{\textit{Comm. Algebra}} \textbf{33} (2005), 1259--1277,
 \href{https://arxiv.org/abs/math.RA/0409275}{math.RA/0409275}.

\bibitem{Cartan}
Cartan E., Les syst\`emes de {P}faff, \`a cinq variables et les \'{e}quations
 aux d\'{e}riv\'{e}es partielles du second ordre, \href{https://doi.org/10.24033/asens.618}{\textit{Ann. Sci. \'{E}cole
 Norm. Sup.~(3)}} \textbf{27} (1910), 109--192.

\bibitem{Gr}
de~Graaf W.A., Classification of solvable {L}ie algebras, \href{https://doi.org/10.1080/10586458.2005.10128911}{\textit{Experiment. Math.}} \textbf{14} (2005), 15--25, \href{https://arxiv.org/abs/math.RA/0404071}{math.RA/0404071}.

\bibitem{Diatta1}
Diatta A., Left invariant contact structures on {L}ie groups,
 \href{https://doi.org/10.1016/j.difgeo.2008.04.001}{\textit{Differential Geom. Appl.}} \textbf{26} (2008), 544--552,
 \href{https://arxiv.org/abs/math.DG/0403555}{math.DG/0403555}.

\bibitem{Vergne1}
Duflo M., Vergne M., Une propri\'{e}t\'{e} de la repr\'{e}sentation coadjointe
 d'une alg\`ebre de {L}ie, \textit{C.~R~Acad. Sci. Paris S\'{e}r. A-B}
 \textbf{268} (1969), A583--A585.

\bibitem{godbillon}
Godbillon C., G\'{e}om\'{e}trie diff\'{e}rentielle et m\'{e}canique analytique,
 Hermann, Paris, 1969.

\bibitem{GozeEllipse}
Goze M., Alg\`ebres de Lie de dimension finie, {R}amm Algebra Center, available
 at \url{http://ramm-algebra-center.monsite-orange.fr}.

\bibitem{GozeCras1}
Goze M., Sur la classe des formes et syst\`emes invariants \`a gauche sur un
 groupe de {L}ie, \textit{C.~R. Acad. Sci. Paris S\'{e}r.~A-B} \textbf{283}
 (1976), A499--A502.

\bibitem{GozeCras2}
Goze M., Mod\`eles d'alg\`ebres de {L}ie frobeniusiennes, \textit{C.~R.~Acad.
 Sci. Paris S\'{e}r.~I Math.} \textbf{293} (1981), 425--427.

\bibitem{BG}
Goze M., Bouyakoub A., Sur les alg\`ebres de {L}ie munies d'une forme
 symplectique, \textit{Rend. Sem. Fac. Sci. Univ. Cagliari} \textbf{57}
 (1987), 85--97.

\bibitem{GH}
Goze M., Haraguchi Y., Sur les {$r$}-syst\`emes de contact, \textit{C.~R.~Acad.
 Sci. Paris S\'{e}r.~I Math.} \textbf{294} (1982), 95--97.

\bibitem{GR-DGA}
Goze M., Remm E., Contact and {F}robeniusian forms on {L}ie groups,
 \href{https://doi.org/10.1016/j.difgeo.2014.05.008}{\textit{Differential Geom. Appl.}} \textbf{35} (2014), 74--94.

\bibitem{GR-Kegel}
Goze M., Remm E., {$k$}-step nilpotent {L}ie algebras, \href{https://doi.org/10.1515/gmj-2015-0022}{\textit{Georgian
 Math.~J.}} \textbf{22} (2015), 219--234, \href{https://arxiv.org/abs/1502.05016}{arXiv:1502.05016}.

\bibitem{Ha}
Hatipoglu C., Injective hulls of simple modules over nilpotent Lie color
 algebras, \href{https://arxiv.org/abs/1411.1512}{arXiv:1411.1512}.

\bibitem{KMS}
Khuhirun B., Misra K.C., Stitzinger E., On nilpotent {L}ie algebras of small
 breadth, \href{https://doi.org/10.1016/j.jalgebra.2015.07.036}{\textit{J.~Algebra}} \textbf{444} (2015), 328--338,
 \href{https://arxiv.org/abs/1410.2778}{arXiv:1410.2778}.

\bibitem{Ki}
Kirillov A.A., Elements of the theory of representations, \href{https://doi.org/10.1007/978-3-642-66243-0}{\textit{Grundlehren
 der Mathematischen Wissenschaften}}, Vol.~220, Springer-Verlag, Berlin~-- New York, 1976.

\bibitem{Ooms}
Ooms A.I., On {F}robenius {L}ie algebras, \href{https://doi.org/10.1080/00927878008822445}{\textit{Comm. Algebra}} \textbf{8}
 (1980), 13--52.

\bibitem{Pin}
Pinczon G., Ushirobira R., New applications of graded {L}ie algebras to {L}ie
 algebras, generalized {L}ie algebras, and cohomology, \textit{J.~Lie Theory}
 \textbf{17} (2007), 633--667, \href{https://arxiv.org/abs/math.RT/0507387}{math.RT/0507387}.

\bibitem{RBreadth}
Remm E., Breadth and characteristic sequence of nilpotent {L}ie algebras,
 \href{https://doi.org/10.1080/00927872.2016.1233238}{\textit{Comm. Algebra}} \textbf{45} (2017), 2956--2966, \href{https://arxiv.org/abs/1605.06583}{arXiv:1605.06583}.

\bibitem{Rfili}
Remm E., On filiform {L}ie algebras. {G}eometric and algebraic studies,
 \textit{Rev. Roumaine Math. Pures Appl.} \textbf{63} (2018), 179--209,
 \href{https://arxiv.org/abs/1712.00318}{arXiv:1712.00318}.

\end{thebibliography}
\end{document}